\documentclass[12pt,a4paper,oneside]{amsart}
\usepackage{amsfonts, amsmath, amssymb, amsthm, amscd}
\usepackage{hyperref}
\hypersetup{
  colorlinks   = true, 
  urlcolor     = blue, 
  linkcolor    = blue, 
  citecolor   = red 
}
\usepackage{anysize}
\marginsize{2cm}{2cm}{1.5cm}{1.5cm}

\usepackage{cmap}
\usepackage[T1]{fontenc}
\usepackage[utf8]{inputenc}
\usepackage[english]{babel}
\usepackage{mathtools}
\usepackage[dvips]{graphicx}
\usepackage{psfrag}
\usepackage{enumerate}
\usepackage{graphicx}

\usepackage{subcaption}

\usepackage{xcolor}
\usepackage{hyperref}
\hypersetup{
	colorlinks   = true, 
	urlcolor     = blue, 
	linkcolor    = blue, 
	citecolor    = red 
}
\newcommand{\Href}[2]{\hyperref[#2]{#1~\ref{#2}}}

\newtheorem{thm}{Theorem}[section]
\newtheorem{lemma}{Lemma}[section]

\newtheorem{conj}{Conjecture}
\newtheorem{cor}{Corollary}[section]

\theoremstyle{definition}
\newtheorem{definition}{Definition}

\theoremstyle{remark}
\newtheorem{remark}{Remark}

\newtheorem{claim}{Claim}

\newtheoremstyle{theoremtr}
  {\topsep}
  {\topsep}
  {}
  {0pt}
  {\bfseries}
  {. }
  { }
  {\thmname{#1}\thmnumber{ #2}\thmnote{ (#3)}}

\theoremstyle{theoremtr}
\newtheorem{transformation}{Transformation}
\numberwithin{equation}{section}

\newcommand{\vol}[1]{\operatorname{vol}\nolimits_{#1}}%
\newcommand{\R}{\mathbb{R}}
\newcommand{\N}{\mathbb{N}}
\newcommand{\iprod}[2]{\left\langle#1,#2\right\rangle}
\newcommand{\cube}{\Box}%
\def\co{\mathop{\rm co}}

\def\span{\mathop{\rm span}}
\def\area{\mathop{\rm Area}}

\newcommand\invisiblesection[1]{%
  \refstepcounter{section}%
  \addcontentsline{toc}{section}{\protect\numberline{\thesection}#1}%
  \sectionmark{#1}}

\title{On the volume of sections of the cube}
\author{Grigory Ivanov\address{G.I.: Inst. of Discrete Mathematics and 
Geometry, TU Wien, Vienna}
\and
Igor Tsiutsiurupa \address{I.T.: Moscow Institute of Physics and Technology,  Institutskii pereulok 9, Dolgoprudny, Moscow
region, 141700, Russia}
}


\begin{document}

\begin{abstract}
We study the properties of the maximal volume $k$-dimensional sections of the $n$-dimensional cube $[-1,1]^n$. We obtain a first order necessary condition  for a $k$-dimensional subspace to be a local maximizer of the volume of such sections, which we formulate in a geometric way.  We estimate  the length of the projection of a vector of the standard basis of $\R^n$ onto a $k$-dimensional subspace that maximizes the volume of the intersection. We find the optimal upper bound on the volume of a  planar section of the cube $[-1,1]^n,$ $n \geq 2.$
\end{abstract}

\maketitle

\bigskip


\bigskip


\section{Introduction}
The problem of volume extrema of the intersection of the standard $n$-dimensional cube $\cube^n = [-1,1]^n$ with a $k$-dimensional linear subspace $H$ has been studied intensively. The tight  lower bound 
for all $n \geq k  \geq 1$ was obtained by J.~Vaaler \cite{vaaler}, he showed that
\[
\vol{k}\cube^k\leq\vol{k}(\cube^n\cap H).
\]
A. Akopyan and R. Karasev \cite{akopyan2019lower} gave a new proof of Vaaler's inequality in terms of waists. A deep generalization  of Vaaler's result for $\ell_p^n$ balls was made by  M. Meyer and A. Pajor \cite{meyer1988sections}.
K. Ball in \cite{ballvolumes}, using his celebrated version of the Brascamb--Lieb inequality, found the following upper bounds
\begin{equation}\label{eq:cube_vol_upper_bound}
\vol{k}(\cube^n \cap H)\leq \left(\frac{n}{k}\right)^{k/2} \vol{k}\cube^k \quad\quad
\text{and} \quad\quad  \vol{k}(\cube^n \cap H)\leq  \left(\sqrt{2}\right)^{n-k} \vol{k}\cube^k.
	\end{equation}
The leftmost inequality here is tight if and only if $k|n$ (see \cite{ellips}), and the rightmost one is tight whenever $2k\geq n$. 
Thus, if $k$ does not divide $n$ and $2k<n,$ the maximal volume of a section of $\cube^n$ remains unknown. 
For the hyperplane case $k = n-1,$ the rightmost inequality in \eqref{eq:cube_vol_upper_bound} was generalized to certain product measures
which include Gaussian type measures by A. Koldobsky and H. K{\"o}nig \cite{konig2012maximal}.

In \cite{crospol}, a tight bound on the volume of a section of $\cube^n$ by a $k$-dimensional linear subspace was conjectured for all $n > k \geq 1.$
Namely, let $C_\cube(n,k) 2^k$ be the maximum volume of a  section of  $\cube^n$ 
by a $k$-dimensional subspace $H$ such that $\cube^n \cap H$ is an affine cube. 
\begin{conj}\label{conj:cube_conjecture}
	The maximal volume of a section of the cube $\cube^n$ by a $k$-dimensional subspace $H$ is attained on subspaces such that the section is an affine cube, i.e.
	\begin{equation*}
		\vol{k}(\cube^n\cap H)\leq C_{\cube}(n,k)\vol{k}\cube^k.
	\end{equation*}
\end{conj}
It is not hard to show that
	\begin{equation}\label{eq:optimal_constant}
		C^2_{\cube}(n,k)=\left\lceil\frac{n}{k}\right\rceil^{n-k\lfloor n/k\rfloor}\left\lfloor\frac{n}{k}\right\rfloor^{k-(n-k\lfloor n/k\rfloor)}.
	\end{equation}
We give a complete description of the set of $k$-dimensional subspaces of $\R^n$ on which $C_{\cube}(n,k)$ is attained and satisfies identity \eqref{eq:optimal_constant} in Lemma \ref{lem:plane_extremizers_cube}.

In this paper, we continue our study of maximizers of
\begin{equation}\label{eq:main_pr_original}
		G(H)=\vol{k}(\cube^n\cap H),    \quad H \in \operatorname{Gr}(n,k) 
		\text{ with }  n \geq k \geq 2.
\end{equation}
Using the approach of \cite{zonotop}, which is described in detail below,
we get a geometric first order necessary condition  for $H$ to be a local maximizer 
of \eqref{eq:main_pr_original}.
\begin{thm}\label{thm:cube_first_order_n_c}
Let $H$ be a local maximizer of \eqref{eq:main_pr_original}, 
$v_i$ be the projection of the $i$-th vector of the standard basis of $\R^n$
onto $H,$ $i \in [n].$ Denote $P= \cube^n \cap H$, we understand $P$ as a $k$-dimensional polytope in $H.$ Then
\begin{enumerate}
	\item\label{item:first_order_section} $P = \bigcap\limits_{[n]} \left\{x\in H: \lvert\langle x, v_i\rangle\rvert \leq1\right\}.$
	\item\label{item:first_order_facet} For every $i \in [n],$ $v_i \neq 0$ and the intersection of $P$ with  the  hyperplane $\{\iprod{x}{v_i} = 1\}$  is a facet of $P.$
	\item\label{item:first_order_centroid} For every $i \in [n],$   the line $\span \{v_i\}$ intersects the boundary of $P$ in the centroid of a facet of $P.$ 
	\item\label{item:first_order_volumes} Let $F$ be a facet of $P.$ Denote $P_F = \co \{0, F\}.$ Then 
		\[
		\frac{\vol{k} P_F}{\vol{k} P} = \frac{1}{2} \frac{\sum_{\star} |v_i|^2}{k},
		\]
		where the summation is over all indices $i \in [n]$ such that the line $\span\{v_i\}$ intersects $F$ in its centroid.
\end{enumerate}
\end{thm}

One of the arguments used by K. Ball to prove the rightmost inequality in \eqref{eq:cube_vol_upper_bound} is that the projection of a vector of the standard basis onto a maximizer of \eqref{eq:main_pr_original} for 
$2k \geq n$ has length at least $\sqrt{2}.$ 
We prove the following extension of this result.
\begin{thm}\label{thm:squared_length_estimate}
	Let $n > k \geq 1$ and $H$ be a global maximizer of \eqref{eq:main_pr_original},
	$v$ be the projection of a vector of the standard basis of $\R^n$
onto $H.$   Then
		\begin{equation*}
			\frac{k}{n+k}\leq |v|^2 \leq\frac{k}{n-k}.
		\end{equation*}
\end{thm}

Using these results and some additional geometric observations, we prove the following. 

\begin{thm}\label{thm:main_result}
	\Href{Conjecture}{conj:cube_conjecture} is true for $n > k=2.$ That is, for any two-dimensional subspace $H \subset \R^n$ the following inequality holds
	\begin{equation*}
		\area(\cube^n\cap H)\leq C_{\cube}(n,2)\vol{2}\cube^2 = 4\sqrt{\left\lceil\frac{n}{2}\right\rceil\left\lfloor\frac{n}{2}\right\rfloor}.
	\end{equation*}
	This bound is optimal and is attained if and only if $\cube^n \cap H$ 
	is a rectangle with  the sides  of lengths 
	$2\sqrt{\left\lceil\frac{n}{2}\right\rceil}$ and 
	$2\sqrt{\left\lfloor\frac{n}{2}\right\rfloor}$.
\end{thm}

\section{Definitions and Preliminaries}\label{sec:def_pre}
For a positive integer $n$, we refer to the set $\{1, 2, \dots,  n \}$ as $[n].$ The standard $n$-dimensional cube  $[-1,1]^n$ is denoted by $\cube^n.$   We use $\iprod{p}{x}$ to denote the \emph{standard inner product}  of vectors $p$ and  $x$ in $\R^n$. For vectors $u,v\in\R^n$, their \emph{tensor product} (or, diadic product) is the 
linear operator on $\R^n$ defined as $(u\otimes v)x=
\iprod{u}{x} v$ for every 
$x\in\R^d.$ The linear hull of a subset $S$ of $\R^n$ is denoted by $\span S$. 
For a $k$-dimensional linear subspace $H$ of $\R^n$ and a body $K \subset H,$ we denote by $\vol{k} K$ the $k$-dimensional volume of $K$. 
The two-dimensional volume of a body $K \subset \R^2$ is denoted by  $\area K.$ 
We denote the identity operator on a linear subspace $H\subset\R^n$ by $I_H$. If $H=\R^k$, we use $I_k$ for convenience.

For a non-zero vector $v \in \R^k,$ we denote by $H_v$ the affine hyperplane $\{x\in \R^k:\iprod{x}{v} = 1\}$, and by $H_v^+$ and $H_v^-$ the half-spaces $\{x\in \R^k:\iprod{x}{v} \le 1\}$  and 
$\{x\in \R^k:\iprod{x}{v} \ge -1\}$, respectively.

It is convenient to identify a section of the cube with a convex polytope in $\R^k$. Let $\{v_1, \dots, v_n\}$ be the projections of the standard basis of $\R^n$ onto $H.$ Clearly,
\begin{equation*}
	\cube^n\cap H=\bigcap_{i=1}^n\{x\in H : \lvert\iprod{x}{v_i}\rvert \leq 1\}.
\end{equation*}
That means that a section of $\cube^n$ is determined by the set of vectors $\{v_i\}_1^n\subset H$, which are the projections of the orthogonal basis. Such sets of vectors have several equivalent description and names. 

\begin{definition}
We will say that an  ordered $n$-tuple of  vectors $\{v_1, \dots, v_n\} \subset H$ is {\it   a tight frame} (or forms a tight frame) in  a vector space $H$ 
if 
\begin{equation}
\label{John_cond_3}
\left. \left(\sum\limits_1^n v_i \otimes v_i \right) \right|_H = I_H,
\end{equation}
where $I_H$ is the identity operator in $H$  and $\left. A\right|_H$ is the restriction of an operator  $A$ onto $H.$ \\
	We  use $\Omega(n,k)$ to denote the set of all tight frames with $n$ vectors in $\R^k.$
\end{definition}
\begin{definition}
	An $n$-tuple of  vectors in a linear space $H$ that spans $H$ is called a \textit{frame}.
\end{definition}

In the following trivial lemma we understand $\R^k \subset \R^n$ as the subspace of vectors, whose last $n-k$ coordinates are zero.
For convenience, we will consider $\{v_i\}_1^n \subset \R^k \subset \R^n$ as $k$-dimensional vectors. 
\begin{lemma}\label{lem:tight_frame_equiv_cond}
The following assertions
 are equivalent:
\begin{enumerate}
	\item \label{ass:eq_cond_it1} the vectors  $\{v_1, \dots, v_n\} \subset \R^k$  form a tight frame in $\R^k$; 
	\item \label{ass:eq_cond_it2} there exists an orthonormal basis $\{ f_1, \ldots, f_n \}$  of  $\R^n$ such that   $v_i$ is the orthogonal projection 
	of $f_i$ onto $\R^k,$ for any $i \in [n];$ 
	\item \label{ass:eq_cond_Gram} 
	$\span\{v_1, \ldots ,v_n\} = \R^k$ and the Gram matrix $\Gamma$ of vectors $\{v_1, \ldots ,v_n\} \subset \R^k$ is the matrix of the projection operator from $\R^n$ onto the linear hull of the rows of the matrix   $M = (v_1, \ldots, v_n).$
	\item \label{ass:eq_cond_it3}  the $k\times n$ matrix $M = (v_1, \ldots, v_n )$ is a sub-matrix of an orthogonal matrix of order $n$.
\end{enumerate} 
\end{lemma}

It follows that the tight frames in $\R^k$ are exactly the projections of orthonormal bases onto $\R^k$. This observation allows us to reformulate the problems in terms of tight frames and associated polytopes in $\R^k$.
Indeed, identifying $H$ with $\R^k,$ we identify the projection of the standard basis onto $H$ with a tight frame  $\{v_1, \dots, v_n\} \subset \Omega(n,k).$ Thus, we identify $\cube^n\cap H$ with the intersection of slabs  of the form $H_{v_i}^+ \cap H_{v_i}^-,$ $i \in [n].$ Vice versa, assertion~\eqref{ass:eq_cond_Gram} gives a way to reconstruct $H$ from a given tight frame $\{v_1, \dots, v_n\}$ in $\R^k.$ 

\begin{definition}\label{def_ass_fr}
We will say that an $n$-tuple $S=\{v_1, \ldots, v_n\}\in\R^k$ generates
\begin{enumerate}
	\item the polytope 
		\begin{equation}\label{generated_section}
			\cube(S)=\bigcap\limits_{i \in [n]}\left( H_{v_i}^+ \cap H_{v_i}^-\right),
		\end{equation}
		which we call the \emph{section of the cube generated by} $S$; 
	\item the matrix  $ \sum\limits_{i \in [n]}  v_i \otimes v_i$. We use $A_S$ to denote this matrix. 
\end{enumerate}
\end{definition}

To sum up, the global extrema of \eqref{eq:main_pr_original} coincide with that of 
\begin{equation}\label{eq:main_problem_cube}
 F(S) = \vol{k}\cube(S), \text{ where }\;  S \in \Omega(n,k)
 \; \text{  with  }\;  n \geq k > 1.
\end{equation}
Moreover, it was shown in \cite{zonotop} that the local extrema of \eqref{eq:main_problem_cube} coincide with that of \eqref{eq:main_pr_original}. However, we note that 
there is an ambiguity when we identify $H$ with $\R^k.$ Any choice of orthonormal basis of $H$ gives its own tight frame in $\R^k $, all of them are isometric but different from each other. It is not a problem as there exists a one-to-one correspondence between $\operatorname{Gr} (n,k)$ and $\frac{\Omega(n,k)}{\operatorname{O}(k)},$ where $\operatorname{O}(k)$ is the orthogonal group in dimension~$k$. And, clearly, $F(S_1) = F(S_2)$ whenever $S_2 = U (S_1)$ for some 
$U \in \operatorname{O}(k).$

In the following lemma, we give a complete description of the set 
$\mathcal{H}$ of $k$-dimensional subspaces of $\R^n$ such that 
$\cube^n \cap H$ is an affine cube and 
$\vol{k} \left(\cube^n \cap H\right) = 2^k C_\cube (n,k)$ for $H \in \mathcal{H}.$
Since it was proven in \cite{crospol}, 
we provide a sketch of its proof in  \Href{Appendix}{sec:appendix}.
\begin{lemma}\label{lem:plane_extremizers_cube}
Constant $C_{\cube} (n,k)$ is given by \eqref{eq:optimal_constant} and is attained on the subspaces given by the following rule.
\begin{enumerate}
\item \label{enum_card_1} We partition $[n]$  into $k$ sets such that the cardinalities of any two sets differ by at most one.
\item \label{enum_card_2} Let  $\{i_1, \ldots, i_\ell \}$ be one of the sets of the partition. Then, choosing arbitrary signs, we write the system of linear equations
$$
	\pm x[i_1] = \ldots = \pm x[i_\ell],
$$
where $x[i]$ denotes the $i$-th coordinate of $x$ in $\R^n.$
\item \label{enum_card_3} Our subspace is the solution of the system of all  equations written for each set of the partition at  step \eqref{enum_card_2}.
\end{enumerate}
\end{lemma}

From now on, we will study properties of the maximizers of \eqref{eq:main_problem_cube} and work with tight frames.

\section{Operations on frames} \label{sec_pertubations}
The following approach to our problem  was proposed in \cite{zonotop} and
used in \cite{crospol} to study the properties of  projections of the standard cross-polytope.

The main idea is to transform a given tight frame $S$ into a new one $S'$ and compare the volumes of the sections of the cube generated by them.  
Since it is not very convenient to transform a given tight frame into another one, we add an intermediate step: we transform a tight frame $S$ into a frame $\tilde{S}$, and then we transform $\tilde{S}$ into a new tight frame $S'$ using a linear transformation.
The main observation here is that we can always transform any frame $\tilde{S} = \{v_1, \ldots, v_n\}$ into a tight frame $S'$ using a suitable linear transformation $L$: $S' = L\tilde{S}  = \{L v_1, \ldots, L v_n\}$.
Equivalently, any non-degenerate centrally symmetric polytope in $\R^k$ is an affine image of a section of a high dimension cube.

For a frame $S$ in $\R^k$, by definition put 
\begin{equation*}
	B_S=A_S^{-\frac{1}{2}} = \left( \sum\limits_{v\in S} v\otimes v\right)^{-\frac{1}{2}}.
\end{equation*}
The operator $B_S$ is well-defined as the condition $\span S = \R^k$ implies that $A_S$ is a positive definite operator. Clearly, $B_S$ maps any frame $S$ to a tight frame:
\begin{equation*}
	\sum_{v\in S} B_S v\otimes B_S v=B_S \left( \sum_{v\in S} v\otimes v\right) B_S^T = B_S A_S B_S = I_k.
\end{equation*}

We obtain the following necessary and sufficient condition  for a tight frame to be a maximizer of \eqref{eq:main_problem_cube}.
\begin{lemma}\label{lem:n_s_cond_cube}
	The maximum of~\eqref{eq:main_problem_cube} is attained at a tight frame $S\in\Omega(n,k)$ iff for an arbitrary frame $\tilde{S}$ in $\R^k$ inequality
		\begin{equation}\label{eq:criterion_inequality_cube}
			\frac{\vol{k}\cube(\tilde{S})}{\vol{k} \cube(S)}\leq\frac{1}{\sqrt{\det A_{\tilde{S}}}}
		\end{equation}
	holds.
\end{lemma}
\begin{proof}
For any frame $\tilde{S}$, we have that $B_{\tilde{S}}\tilde{S}$ is a tight frame and 
$\vol{k}\cube(B_{\tilde{S}}\tilde{S})=\vol{k}\cube(\tilde{S})/\det B_{\tilde{S}}.$ 
The maximum of~\eqref{eq:main_problem_cube} is attained at a tight frame $S$ iff 
$\vol{k}\cube(B_{\tilde{S}}\tilde{S})\leq\vol{k}\cube(S)$ 
for an arbitrary frame $\tilde{S}.$ Hence
	\begin{equation*}
	1\ge\frac{\vol{k}\cube(B_{\tilde{S}}\tilde{S})}{\vol{k}\cube(S)} = \frac{\vol{k}\cube(\tilde{S})}{\det B_{\tilde{S}}}\frac{1}{\vol{k}\cube(S)}= 
	\frac{\vol{k}\cube(\tilde{S})}{\vol{k}\cube(S)} \sqrt{\det A_{\tilde{S}}}
	\end{equation*}
Dividing by $\sqrt{\det A_{\tilde{S}}}$, we obtain inequality \eqref{eq:criterion_inequality_cube}.
\end{proof}
Clearly, if $\tilde{S}$ in the assertion of \Href{Lemma}{lem:n_s_cond_cube} is close to $S,$ then the tight frame $B_{\tilde{S}} \tilde{S}$ is close to $S$ as well. Therefore, inequality \eqref{eq:criterion_inequality_cube} gives a necessary condition for local maximizers of \eqref{eq:main_problem_cube}. Let us illustrate how we will use it. 

Let $S$ be an extremizer of \eqref{eq:main_problem_cube}, and
$T$ be a map from a subset of $\Omega(n,k)$ to the set of frames in $\R^k$.
In order to obtain properties of extremizers, we  consider a composition of two operations:
	\begin{equation*}\label{eq:framization2}
		S \quad \stackrel{T}{\longrightarrow} \quad \tilde{S} \quad \stackrel{B_{\tilde{S}}}{\longrightarrow} \quad S',
	\end{equation*}  
where  $B_{\tilde{S}}$ is as defined above. For example, see Figure \ref{fig:example_T_v_to_zero}, where  $T$ is the operation of replacing a vector $v$ of $S$ by the origin.
\begin{figure}[h!]
	\includegraphics[scale=1.2]{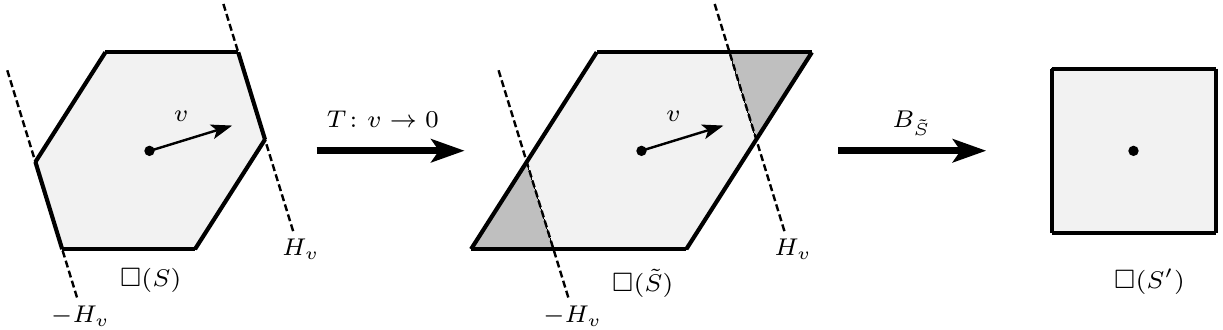}
	\caption{Here we map one vector to zero.}\label{fig:example_T_v_to_zero}
\end{figure}
 
\noindent

Choosing  a  simple operation $T,$   we may calculate the left-hand side of 
\eqref{eq:criterion_inequality_cube} in some geometric terms. We consider several simple operations:  Scaling one or several vectors, mapping one vector to the origin, 
mapping one vector to another.    
On the other hand, the determinant in the   right-hand  side of 
\eqref{eq:criterion_inequality_cube}  can be calculated for the operations listed above.

In particular, the following first-order approximation of the determinant was obtained by the author in 
\cite[Theorem 1.2]{zonotop}. We provide a sketch of its proof in  \Href{Appendix}{sec:appendix}.
\begin{lemma}\label{lem:der_det_identity}
	Let $S=\{v_1,\ldots,v_n\}\subset\R^k$ be a tight frame and the $n$-tuple $\tilde{S}$ be obtained from $S$ by substitution $v_i\to v_i+tx_i$, where $t\in\R$, $x_i\in\R^k$, $i\in[n]$. Then
		\begin{equation*}
			\sqrt{\det A_{\tilde{S}}}=1+t\sum_{i=1}^n\langle x_i, v_i\rangle+o(t).
		\end{equation*}
\end{lemma}

We state as  lemmas several technical facts from linear algebra that will be used later.

\begin{lemma}\label{lem:determinant_for_substitution}
Let $A$ be a positive definite operator on $\R^k$. For any $u \in \R^k,$ we have
\[
\det \left(A \pm u\otimes u \right) =(1 \pm |A^{-1/2} u|^2) \det A.
\] 
\end{lemma}
\begin{proof}
We have
\[
\det \left(A \pm u\otimes u \right) = \det A \cdot \det\left(I_k \pm 
A^{-1/2} u \otimes A^{-1/2} u\right).
\]
Diagonalizing the operator $I_k \pm A^{-1/2} u \otimes A^{-1/2} u,$ we see that its determinant equals $1 \pm |A^{-1/2} u|^2.$ This completes the proof.
\end{proof}

For any $n$-tuple $S$ of vectors of $\R^k$ with $v \in S,$  we use $S\setminus v$ to denote the $(n-1)$-tuple of vectors obtained from $S$ by removing the first occurrence of $v$ in $S.$

\begin{lemma}\label{lem:geom_sense_B_S_v}
Let $S\in\Omega(n,k)$ and $v \in S$ be a vector such that $|v| <1.$ Then $S\setminus v$ is a frame in $\R^k$ and $B_{S \setminus v}$ is the stretch of $\R^k$ by the factor $(1-|v|^2)^{-1/2}$ along  $\span\{v\}.$
In particular, for any $u \in \R^k,$ we have $|B_{S \setminus v} u|\geq|u| .$
\end{lemma}
\begin{proof}
Since $A_{S\setminus v}=I_k-v\otimes v > 0,$ we have that $S\setminus v$ is a frame in $\R^k$.  Clearly, $A_{S\setminus v}$ stretches the space  by the factor $(1-|v|^2)$ along  $\span \{v\}
.$
Therefore,  the operator $B_{S\setminus v} = A_{S\setminus v}^{-1/2}$ 
stretches the space by the factor $(1-|v|^2)^{-1/2}$ along the same direction.
\end{proof}

\section{Properties of a global maximizer}
\Href{Theorem}{thm:squared_length_estimate} is formulated  in terms of subspaces. For the sake of convenience, we introduce  its equivalent  reformulation in terms of tight frames.  
\begin{thm}[Frame version of Theorem \ref{thm:squared_length_estimate}]
\label{thm:squared_length_estimate_frame}
Let $S \in \Omega(n,k)$ be a global  maximizer of \eqref{eq:main_problem_cube} for $n > k \geq 1,$ let $v \in S.$ Then 
		\begin{equation}\label{eq:squared_length_estimate}
			\frac{k}{n+k}\leq |v|^2 \leq\frac{k}{n-k}.
		\end{equation}
\end{thm}
\begin{proof}
	The definition of tight frame implies that
		$\sum_{v\in S}|v|^2=k.$
	Hence there is a vector $u\in S$ such that $|u|^2\leq k/n$ and a vector $w\in S$ such that $|w|^2\geq k/n$.

	We start with the rightmost inequality in~\eqref{eq:squared_length_estimate}. Let $\tilde{S}$ be the $n$-tuple obtained from $S$ by substitution $u\to v$. Since $|u|^2\leq k/n<1$, we have $A_{\tilde{S}}=I_k-u\otimes u+v\otimes v\geq I_k-u\otimes u>0$. Hence $\tilde{S}$ is a frame in $\R^k$. By identity~\eqref{generated_section},  inclusion $\cube(S)\subset\cube(\tilde{S})$ holds. Therefore, $\vol{k}\cube(S)\leq\vol{k}\cube(\tilde{S})$. Using \Href{Lemma}{lem:n_s_cond_cube} and \Href{Lemma}{lem:determinant_for_substitution}, we obtain
		$$1\geq\det A_{\tilde{S}}=(1+|B_{S\setminus u}v|^2)(1-|u|^2). $$
	By \Href{Lemma}{lem:geom_sense_B_S_v},  the right-hand side of this inequality is at least $(1+|v|^2)(1-|u|^2).$ Therefore, 
		$$|v|^2\leq\frac{|u|^2}{1-|u|^2}\leq\frac{k}{n-k}.$$

	Let us prove the leftmost inequality in~\eqref{eq:squared_length_estimate}. There is nothing to prove if $|v|^2 \geq k/n.$ Assume that $|v|< k/n$. Let $\tilde{S}$ be the $n$-tuple obtained from $S$ by substitution $v\to w$. Since $A_{\tilde{S}}=I-v\otimes v+u\otimes u>0,$ $\tilde{S}$ is a frame in $\R^k$. By identity~\eqref{generated_section}, the inclusion $\cube(S)\subset\cube(\tilde{S})$ holds. Therefore, $\vol{k}\cube(S)\leq\vol{k}\cube(\tilde{S})$. Using \Href{Lemma}{lem:n_s_cond_cube} and \Href{Lemma}{lem:determinant_for_substitution}, we obtain
		$$1\geq\det A_{\tilde{S}}=(1+|B_{S\setminus v}w|^2)(1-|v|^2).$$
	Again, by \Href{Lemma}{lem:geom_sense_B_S_v}, the right-hand side of this inequality is at least $(1+|w|^2)(1-|v|^2).$ It follows that
		$$|v|^2\geq\frac{|w|^2}{1+|w|^2}\geq\frac{k}{n+k}.$$
This completes the proof.
\end{proof}
Clearly, \Href{Theorem}{thm:squared_length_estimate_frame}
 implies  \Href{Theorem}{thm:squared_length_estimate}.

These theorems can be sharpened in the planar case.
\begin{lemma}\label{lem:squared_lengths_best_bounds}
Let  $S \in \Omega(n,2)$ be a maximizer of \eqref{eq:main_problem_cube} for $k = 2$ and  $n \geq 3,$    let $v\in S.$ Then
	\begin{equation}\label{eq:vec_lengths_planar_case}
		\frac{2}{n+1} \leq |v|^2 \leq \frac{2}{n-1}.
	\end{equation}
\end{lemma}
\begin{proof}
	By \Href{Lemma}{lem:plane_extremizers_cube},  we have
	\begin{equation}\label{eq:our_bound_cube_planar}
		\area \cube(S) \geq  4\sqrt{\lceil n/2\rceil\lfloor n/2\rfloor}.
	\end{equation}
	Let us prove  the leftmost inequality in \eqref{eq:vec_lengths_planar_case}. It is trivial if $|v| \geq 2/n.$ Assume that $|v| <  2/n .$ Let $S'\in\Omega(n-1,2)$ be a maximizer of~\eqref{eq:main_problem_cube}. By Ball's inequality \eqref{eq:cube_vol_upper_bound}, we have
	\begin{equation}\label{eq:balls_2dim_cube}
		\area \cube(S') \leq 2(n-1). 
	\end{equation}
	Consider $S \setminus v.$ It is a frame by \Href{Lemma}{lem:geom_sense_B_S_v}.     Then, by \Href{Lemma}{lem:n_s_cond_cube} and \Href{Lemma}{lem:determinant_for_substitution}, we get
		$$\frac{\area\cube(S \setminus v)}{\area\cube(S')}\leq\frac{1}{\sqrt{\det A_{S\setminus v}}}=\frac{1}{\sqrt{1-|v|^2}}.$$
	By identity~\eqref{generated_section}, we have $\cube(S)\subset\cube(S \setminus v)$. By this and by inequalities \eqref{eq:balls_2dim_cube} and \eqref{eq:our_bound_cube_planar}, we get
		$$\frac{\area\cube(S \setminus v)}{\area\cube(S')} \geq \frac{\area\cube(S)}{\area\cube(S')}\geq \frac{4\sqrt{\lceil n/2\rceil\lfloor n/2\rfloor}}{2(n-1)}.$$
	Combining the last two  inequalities, we obtain
		$$|v|^2 \geq 1-\left(\frac{\area\cube(S')}{\area\cube(S)}\right)^2\geq 1-\left(\frac{n-1}{2\sqrt{\lceil n/2\rceil\lfloor n/2\rfloor}}\right)^2\geq\frac{2}{n+1}.$$
		
	We proceed with the rightmost inequality in \eqref{eq:vec_lengths_planar_case}. Let $S'\in\Omega(n+1,2)$ be a maximizer of~\eqref{eq:main_problem_cube}. 
	By Ball's inequality \eqref{eq:cube_vol_upper_bound}, we have
	\begin{equation}\label{eq:balls_2dim_cube_upper}
		\area \cube(S') \leq 2(n+1). 
	\end{equation} 
	We use $\tilde{S}$ to denote the $(n+1)$-tuple  obtained from $S$ by concatenating $S$ with the vector $v.$ Since $S$ is a frame in $\R^2$, $\tilde{S}$ is a frame in $\R^2$ as well. By \Href{Lemma}{lem:n_s_cond_cube} and \Href{Lemma}{lem:determinant_for_substitution}, we get
		$$\frac{\area\cube(\tilde{S})}{\area\cube(S')}\leq\frac{1}{\sqrt{\det A_{\tilde{S}}}}=\frac{1}{\sqrt{1+|v|^2}}.$$
	By identity~\eqref{generated_section}, we have $\cube(S)=\cube(\tilde{S}).$  By this and by inequalities \eqref{eq:balls_2dim_cube_upper} and \eqref{eq:our_bound_cube_planar}, we get 
	$$\frac{\area\cube(S)}{\area\cube(S')}=
			\frac{\area\cube(\tilde{S})}{\area\cube(S')} \geq 
			\frac{4\sqrt{\lceil n/2\rceil\lfloor n/2\rfloor}}{2(n+1)}.$$
	Combining the last two inequality, we obtain
		$$|v|^2\leq\left(\frac{\area\cube(S')}{\area\cube(S)}\right)^2-1\leq\left(\frac{n+1}{2\sqrt{\lceil n/2\rceil\lfloor n/2\rfloor}}\right)^2-1\leq\frac{2}{n-1}.$$
\end{proof}
\begin{remark}
It is possible to sharpen inequality \eqref{eq:squared_length_estimate} for $k > 2$ and $n > 2k$ using the same approach as in \Href{Lemma}{lem:squared_lengths_best_bounds}.
The idea is to remove $n \!\!\mod k$ from or add $ n - (n \!\!\mod k)$ vectors to a maximizer and compare the volume of a section of the cube generated by the new frame with the Ball bound \eqref{eq:cube_vol_upper_bound}.
However, it doesn't give a substantial improvement.
\end{remark}

\section{Local properties}
In this section, we prove some properties of the local maximizers of \eqref{eq:main_problem_cube}. We will perturb  facets of $\cube(S)$ of a local maximizer $S$ (that is, we will perturb the vectors of $S$ in a specific way corresponding to a perturbation of some facets of the polytope $\cube(S)$).
To this end, we need to recall some general properties  of  polytopes connected to  perturbations of  a half-space supporting  a polytope in its facet.
\subsection{Properties of polytopes}
Recall that a point $c$ is the \textit{centroid} of a facet $F$ of a polytope $P \subset \R^k$ if 
\begin{equation}
\label{eq:centroid_of_a_facet}
c = \frac{1}{\vol{k-1} F} \int\limits_{F} x d\lambda,
\end{equation}

where $d\lambda$ is the standard Lebesgue measure on the hyperplane containing $F.$

For a set $W \subset \R^k,$ we use $P(W)$ to denote the polytopal set $ \bigcap\limits_{w \in W} H_w^{+}.$ 
Let $W$ be a set of pairwise distinct vectors such that
\begin{itemize}
\item  the set $P(W)$ is a polytope;
\item  for every $w \in W,$ the hyperplane $H_w$ supports $P(W)$ in a facet of $P(W).$
\end{itemize} 
That is, $W$ is the set of scaled outer normals of $P(W).$  
Denote $P = P(W).$
We fix $w \in W$ and the facet $F =  P \cap H_w$ of $P.$
Let $c$ be the centroid of $F.$

\begin{transformation}\label{trans:shift_facet}
We will ``shift'' a facet of a polytope parallel to itself.
Let $W'$  be obtained  from $W$ by substitution $w \to w + h \frac{w}{|w|},$ where $h \in \R.$ Denote $P' = P(W').$ That is, the polytopal set $P'$ is obtained from $P$ by the shift  of the half-space $H_w^+$ by $h$ in the direction of  its  outer normal.  
By the celebrated Minkowski existence  and uniqueness theorem for convex polytopes (see, for example,  \cite[Theorem 18.2]{gruber}), we have
	\begin{equation}\label{eq:vol_normal_variation}
		\vol{k} P'-\vol{k} P=h\vol{k-1}F+o(h).
	\end{equation}
	\begin{figure}[!ht]
		\includegraphics[scale=1.2]{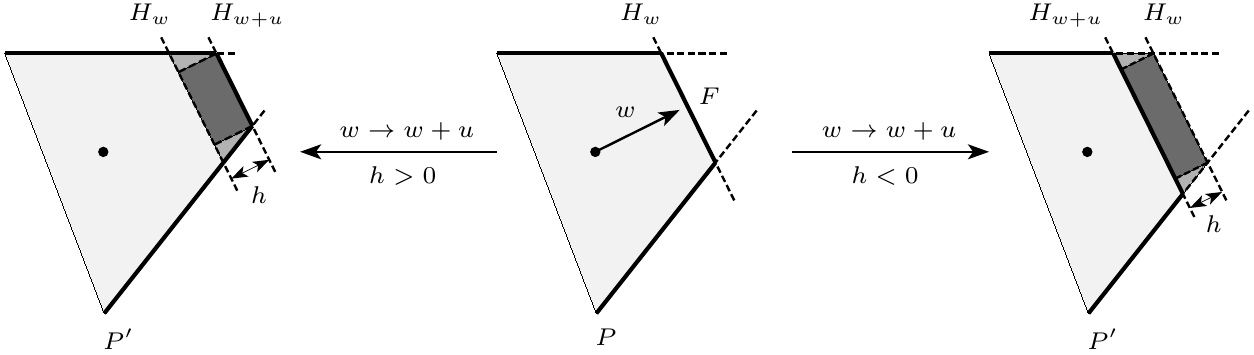}
		\caption{Parallel shift of the facet $F$ by the vector $u=hw/|w|$}\label{fig:facet_shifting}
	\end{figure}
\end{transformation}

\begin{transformation}\label{trans:rotate_facet}
We will rotate a facet around a codimension two subspace.
Let $u$ be a unit vector orthogonal to $w.$ 
Define $c_w = H_w \cap \span \{ w \}$ and $L_u = H_w \cap (u^\perp + c_w).$ Note that $L_u$ is a codimension two affine subspace of $\R^k$ and an affine hyperplane in $H_w.$ Clearly, for any non-zero $t \in \R,$ $L_u  = H_w \cap H_{w + tu} = \left\{x \in H_w \mid \iprod{x-c(w)}{u}=0 \right\}.$
Thus, $L_u$ divides $F$ into two parts \[
F^+ = F \cap \left\{x \in H_w \mid \iprod{x-c(w)}{u} \geq 0 \right\}
\;\text{ and  } \;
F^- = F \cap \left\{x \in H_w \mid \iprod{x-c(w)}{u} \leq 0 \right\}
\] 
(one of the sets $F^+$ or $F^-$ is empty if $c(w) \notin F$).
Let $\alpha \in (-\pi/2, \pi/2)$ be the oriented angle  between hyperplanes $H_{w}$ and $H_{w+tu}$ such that $\alpha$ is positive for positive $t.$

Let $W'$  be obtained  from $W$ by substitution $w \to w + t u,$ where $t \in \R$ and $u$ is a unit vector orthogonal to $w.$ Denote $P' = P(W').$ Thus, for a sufficiently small $|t|,$ the polytopal set $P'$ is a polytope obtained from $P$ by the rotation of the half-space  $H_w^+$ around the codimension two affine subspace $L_u$ by some angle $\alpha = \alpha(t).$
Clearly, in order to calculate the volume of $P',$ we need to subtract from $\vol{k} P$ the volume of the subset of $P$ that is  above $H_{w+tu}$ and to add to   $\vol{k} P$ the volume of the subset of $P'$ that is above $H_w.$  Formally speaking, 
denote 
\[
 Q^+ = \begin{cases}
P' \cap \left(\R^k \setminus H_w^+\right) & \text{for }\;\; \alpha \geq 0 \\
P \cap \left(\R^k \setminus H_{w+tu}^+\right) & \text{for }\;\; \alpha < 0
\end{cases}
\;\text{ and } \;
 Q^- = \begin{cases}
P \cap \left(\R^k \setminus H_{w+tu}^+\right) & \text{for }\;\; \alpha \geq 0 \\
P' \cap \left(\R^k \setminus H_{w}^+\right) & \text{for }\;\; \alpha < 0
\end{cases}
\] 
Then, we have (see  \Href{Figure}{fig:facet_rotation})
\begin{equation}
\label{eq:operation_polytope_2_delta_vol_1}
\vol{k} P'-\vol{k} P = \operatorname{sign} \alpha 
\left( \vol{k} Q^+ -\vol{k} Q^-\right).
\end{equation}

There is a nice approximation for $\vol{k} Q^+ -\vol{k} Q^-.$
Let $C_\alpha^+$ (resp., $C_\alpha^-$) be the set swept out by  
$F^{+}$ (resp., $F^-$) while rotating around $L_u$ by the angle $\alpha.$
By  routine, 
\[
\vol{k} C_\alpha^{+} = |\alpha| \int\limits_0^{+\infty} r \vol{k-2}
\left(F \cap (L_u + r u) \right) dr
\]
\[
\left( \text{resp.,} \quad 
\vol{k} C_\alpha^{-} = |\alpha| \int\limits_0^{+\infty} r \vol{k-2}
\left(F \cap (L_u - r u) \right) dr =
-|\alpha| \int\limits_{-\infty}^{0} r \vol{k-2}
\left(F \cap (L_u + r u) \right) dr  \right). 
\]
We claim here without  proof that
\[
\vol{k} Q^+ = \vol{k} C_\alpha^{+} + o(\alpha) \quad \text{ and} \quad
\vol{k} Q^- = \vol{k} C_\alpha^{-} + o(\alpha).
\]
By this and identity \eqref{eq:operation_polytope_2_delta_vol_1}, we obtain 
\begin{equation*}
\vol{k} P'-\vol{k} P = \alpha \left(\vol{k} C_\alpha^{+} - \vol{k} C_\alpha^{-} \right) + o(\alpha) = \alpha \int\limits_{\R} r \vol{k-2}
\left(F \cap (L_u + r u) \right) dr + o(\alpha).
\end{equation*}
By this and by \eqref{eq:centroid_of_a_facet}, we get
\begin{equation*}
\vol{k} P'-\vol{k} P =\alpha  \iprod{c - c_w}{u} \vol{k-1} F   + o(\alpha).
\end{equation*}
	Since $w$ and $u$ are orthogonal, we have
	\begin{equation*}
		\alpha=\arctan\frac{|u|}{|w|}t=\frac{|u|}{|w|}t+o(t).
	\end{equation*}
Finally,  we obtain
\begin{equation}
\label{eq:operation_polytope_2_final_destination}
\vol{k} P'-\vol{k} P  = \frac{\vol{k-1} F }{|w|}  \iprod{c - c_w}{u} t + o(t).
\end{equation}
	\begin{figure}[!ht]
	\includegraphics[scale=1.2]{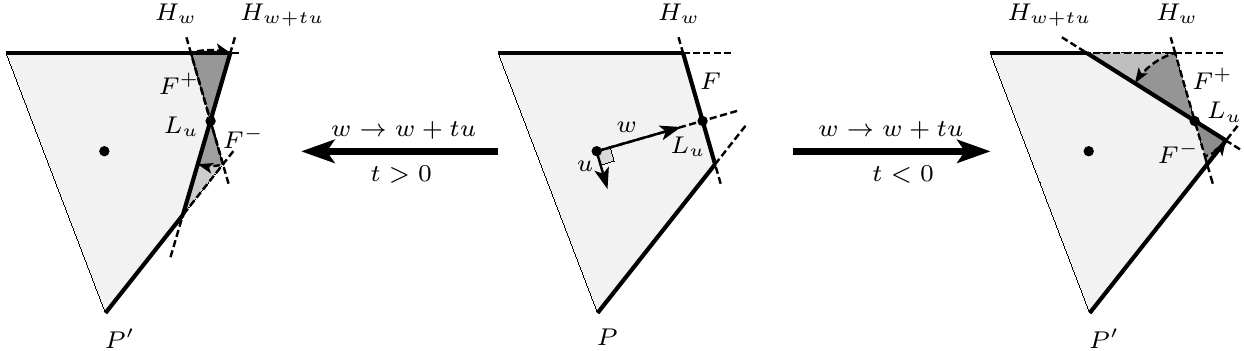}
	\caption{A rotation of the facet $F$ around $L_u$}\label{fig:facet_rotation}
	\end{figure}
\end{transformation}

\subsection{Local properties of sections of the cube}
Let $S$ be a frame in $\R^k.$
For every $v \in S,$ we denote the set  $H_v\cap \cube(S)$ by $F_v.$   We say that $v\in S$ \emph{corresponds to a facet} $F$ of $\cube(S)$ if either $F = F_v$ or $F = - F_v.$   Clearly, if some vectors of $S$ correspond to the same facet of $\cube(S),$ then they are equal up to a sign.
  For a given frame $S$ in $\R^k$ and $u \in \R^k,$ a facet $F$ of $\cube(S)$ and a vector $u \in \R^k,$ we define an $F$-substitution in the direction $u$ as follows: 
\begin{itemize}
\item each vector $v$ of $S$  such that $F \subset H_v$ is substituted by $v+u;$   
\item each vector $v$ of $S$  such that  $-F \subset H_v$ is substituted by $v-u;$   
\item all other vectors of $S$ remain the same.
\end{itemize}
In order to prove \Href{Theorem}{thm:cube_first_order_n_c}, we will use  $F$-substitutions. 

At first, we simplify the structure of a local maximizer.
\begin{lemma}\label{lem:hyperplane_meets_facet}
Let $S$ be a local maximizer of \eqref{eq:main_problem_cube} and  $v\in S.$
Then $F_v$ is a facet of $\cube(S).$
\end{lemma}
\begin{proof}
Let
$K$ be a convex body in $\R^k,$ then its polar body  is defined by
\[ 
 \left\{ y\in\R^k : \iprod{y}{x} \leq 1 \text{ for all } x \in K \right\}.
\]
Since $\cube(S)$ is the intersection of half-spaces of the form
$\left\{\iprod{w}{x} \leq 1 \right\}$ with $w \in \pm S,$
we have that $\co \{\pm S\}$ is polar to $\cube(S).$
By the duality argument, it suffices to prove that  $v\in S$ is a vertex of the polytope $\co\left\{ \pm S\right\}.$ Assume that $v$ is not a vertex of $\co\left\{ \pm S\right\}.$ 

Clearly, $v \in \co \left\{\pm \left(S \setminus v\right)\right\}$ and $v$ is not a vertex of the polytope $\co \left\{\pm \left(S \setminus v\right)\right\}.$
Therefore we have that  $\span \{S\setminus v \} = \span S = \R^k.$ That is, $S \setminus v$ is a frame in $\R^k.$ 
Since $B_{S \setminus v}$ is a nondegenerate linear transformation, $B_{S \setminus v} v$ is not a vertex of the polytope 
$ \co \left\{\pm B_{S \setminus v}(S \setminus v)\right\}.$
By this and by the triangle inequality,  there is a vertex  $u$ of 
$ \co  \{\pm S\}$ such that $ u\in S$ and  $|B_{S\setminus v}v|<|B_{S\setminus v}u|$.

Denote by $\tilde{S}$ the $n$-tuple obtained from $S$ by substitution $v\to v+t(u-v)$, where $t\in(0,1]$. Since $A_{\tilde{S}}\geq A_{S\setminus v}>0$, $\tilde{S}$ is a frame in $\R^k$. By the choice of $u$ and identity~\eqref{generated_section}, we have $\cube(\tilde{S})=\cube(S)$. Hence $\vol{k}\cube(\tilde{S})=\vol{k}\cube(S)$. \Href{Lemma}{lem:n_s_cond_cube} implies that $\det A_{\tilde{S}}\leq 1$.

On the other hand, by \Href{Lemma}{lem:determinant_for_substitution}, we have
	$$\det A_{\tilde{S}}=(1+|B_{S\setminus v}(v+t(u-v))|^2)(1-|v|^2).$$
Inequality $|B_{S\setminus v}v|<|B_{S\setminus v}u|$ implies that $|B_{S\setminus v}v|<|B_{S\setminus v}(v+t(u-v))|$. By this and by \Href{Lemma}{lem:geom_sense_B_S_v}, we conclude that $\det A_{\tilde{S}}>1$. This is a contradiction. Thus, $v$ is a vertex of $\co\{ \pm S\}.$ The lemma is proven.
\end{proof}

As an immediate corollary of \Href{Lemma}{lem:hyperplane_meets_facet} and by the standard properties of polytopes, we have the following statement.

\begin{cor}\label{cor:perturbations_explained}
	Let $S$ be a local maximizer of \eqref{eq:main_problem_cube} and  $v \in S.$ Let $\tilde{S}$ be the $n$-tuple obtained from $S$ by $F_v$-substitution in the direction $t u$ with $t\in \R$ and $u \in \R^k.$
	Then, for a sufficiently small $|t|,$ $\tilde{S}$ is a frame, the vector $v + ut$ corresponds to a facet of $\cube(\tilde{S})$ and the polytopes $\cube(S)$ and $\cube (\tilde{S})$ have the same combinatorial structure. Moreover, $\vol{k} \cube(\tilde{S})$ is a smooth function of $t$ at $t=0.$
\end{cor}

In the following two lemmas, we will perturb a local maximizer by making $F$-substitutions. Geometrically speaking, making an $F$-substitution in the direction $tu$  with $u \in \R^k $ and  $t \in \R,$ we move the opposite facets $F$ and $-F$ of a local maximizer in a symmetric way. Thus, for a sufficiently small $t,$  perturbations  of the facets $F$ and $-F$ are independent.

\begin{lemma}\label{lem:necessary_condition_first_order}
Let $S$ be a local maximizer of \eqref{eq:main_problem_cube}. Let  $v\in S$ and  $d$ be the number of the vectors of $S$ that correspond to $F_v.$ Then
	\begin{equation}\label{eq:necessary_condition_first_order}
		\frac{2}{|v|}\vol{k-1} F_v=d|v|^2\vol{k}\cube(S).
	\end{equation}
\end{lemma}
\begin{proof}
	Denote by $\tilde{S}$ the $n$-tuple obtained from $S$ by $F_v$-substitution in the direction $tv$ with $t \in \R.$ 
	Thus, we apply \Href{Transformation}{trans:shift_facet}  to the facets $\pm F_v$ of $\cube(S).$    
	 By \Href{Lemma}{lem:n_s_cond_cube}, we have 
\begin{equation}
\label{eq:ns_cond_lem_shift_plane}
	\frac{\vol{k}\cube(\tilde{S})}{\vol{k} \cube(S)}\leq\frac{1}{\sqrt{\det A_{{\tilde{S}}}}}.
\end{equation}

By \Href{Lemma}{lem:der_det_identity} and  \Href{Corollary}{cor:perturbations_explained},  both sides of this inequality are smooth as functions of $t$ in a sufficiently small neighborhood of  $t_0 =0.$
	Consider the Taylor expansions of both sides of  inequality \eqref{eq:ns_cond_lem_shift_plane} as functions of $t$ about $t_0=0.$ 
	 
	By \Href{Lemma}{lem:determinant_for_substitution}, $\det A_{\tilde{S}}=1+d(2t+t^2)|v|^2$. Hence
	\begin{equation}\label{eq:det_first_order_shrinking}
		\frac{1}{\sqrt{\det A_{\tilde{S}}}}=1-td|v|^2+o(t).
	\end{equation} 
Geometrically speaking, we shift the half-space $H_v^+$ (resp., $H_v^-$) by
	\[
	h=\frac{1}{|(1+t)v|} - \frac{1} {|v|} = - \frac{t}{|v|}+o(t)
	\]
	in the directions of its outer normal.
	By this and by \eqref{eq:vol_normal_variation}, we obtain
	\begin{equation}\label{eq:volume_change_shrink_facet}
	\vol{k}\cube(\tilde{S})-\vol{k}\cube(S)=-\frac{2t}{|v|} \vol{k-1}F_v +  o(t).
	\end{equation}
Using identities \eqref{eq:volume_change_shrink_facet} and \eqref{eq:det_first_order_shrinking} in \eqref{eq:ns_cond_lem_shift_plane} , we get
\[
1 -  \frac{2t}{|v|} \frac{\vol{k-1} F_v}{\vol{k}\cube(S)}\leq
1 - d |v|^2 t + o(t).
\]
	
Since $\tilde{S} = S$ for $t=0$ and  the previous inequality holds for all $t \in (-\varepsilon, \varepsilon)$ for a sufficiently small $\varepsilon,$  the coefficients of $t$ in both sides of the previous inequality coincide. That is, 
	\begin{equation*}
		\frac{2}{|v|}\frac{\vol{k-1} F_v}{\vol{k}\cube(S)}=d|v|^2.
	\end{equation*}
This completes the proof.
\end{proof}

\begin{lemma}\label{lem:intersecting_in_barycenter}
Let $S$ be a local maximizer of \eqref{eq:main_problem_cube} and  $v\in S.$ Then the line $\span\{v\}$ intersects the hyperplane $H_v$ in the centroid of the facet $F_v.$
\end{lemma}
\begin{proof}
Denote the centroid of $F_v$ by $c$ and  let $c_v = \span \{v\} \cap H_v.$ 
Fix a unit vector $u$ orthogonal to $v.$
	Denote by $\tilde{S}$ the $n$-tuple obtained from $S$ by $F_v$-substitution in the direction $tu$ with $t \in \R.$ 
Thus, we apply \Href{Transformation}{trans:rotate_facet}  to the facets $\pm F_v$ of $\cube(S).$    

 By \Href{Lemma}{lem:n_s_cond_cube}, we have 
\begin{equation}
\label{eq:ns_cond_lem_rotate_plane}
	\frac{\vol{k}\cube(\tilde{S})}{\vol{k} \cube(S)}\leq\frac{1}{\sqrt{\det A_{\tilde{S}}}}.
\end{equation}

By \Href{Lemma}{lem:der_det_identity} and  \Href{Corollary}{cor:perturbations_explained}, both sides of this inequality are smooth as functions of $t$ in a sufficiently small neighborhood of  $t_0 =0.$
	Consider the Taylor expansions of both sides of  inequality \eqref{eq:ns_cond_lem_shift_plane} as functions of $t$ about $t_0=0.$ 

By \eqref{eq:operation_polytope_2_final_destination}, we obtain
	\begin{equation*}
		\vol{k}\cube(\tilde{S})-\vol{k}\cube(S)= C \iprod{c - c_v}{u} t + o(t),
	\end{equation*}
	where $C = 2\vol{k-1}F_v/|v| > 0.$
	By \Href{Lemma}{lem:der_det_identity}, $\sqrt{\det A_{\tilde{S}}}=1+o(t).$
Therefore, inequality \eqref{eq:ns_cond_lem_rotate_plane} takes the following form
\[
1 + \frac{C}{\vol{k} \cube(S)}  \iprod{c - c_v}{u} t + o(t) \leq 1 + o(t).
\]
Since $\tilde{S} = S$ for $t=0$ and  the previous inequality
 holds for all $t \in (-\varepsilon, \varepsilon)$ for a sufficiently small $\varepsilon,$  the coefficients of $t$ in both sides of the previous inequality coincide. That is,  we conclude 
	\begin{equation*}
		\iprod{c - c_v}{u} = 0.
	\end{equation*}
Since $c, c_v \in H_v$ and the last identity holds for all unit vectors parallel to $H_v,$ it follows that  $c = c_v.$
The lemma is proven.    
\end{proof}

As a simple consequence of \Href{Lemma}{lem:intersecting_in_barycenter}, we obtain the following result for the planar case.
\begin{thm}\label{thm:maximizer_cyclic}
Let $S \in \Omega(n,2)$ be a local maximizer of \eqref{eq:main_problem_cube} for $k=2.$ Then, the polygon $\cube(S)$ is cyclic. That is, there is a circle that passes through all the vertices of $\cube(S)$.
\end{thm}
\begin{proof}
	Denote the origin by $o$. Let $ab$ be an edge of $\cube(S)$ and $oh$ be the altitude of the triangle $abo$. By \Href{Lemma}{lem:intersecting_in_barycenter}, $h$ is the midpoint of $ab.$  Hence, the triangle $abo$ is isosceles and $ao = bo.$  It follows that $\cube(S)$ is cyclic.
\end{proof}

We are ready to give a proof of \Href{Theorem}{thm:cube_first_order_n_c}.

\begin{proof}[Proof of \Href{Theorem}{thm:cube_first_order_n_c}]
Recall that any identification of $H$ with $\R^k$ identifies the projections of the standard basis $\{v_1, \dots, v_n\}$ with a tight frame, denoted by $S,$ that is a local maximizer of \eqref{eq:main_problem_cube}.

Next, assertion~\ref{item:first_order_section} is trivial and holds for any section of the cube.
 Assertion~\ref{item:first_order_facet}  and  assertion~\ref{item:first_order_centroid}   are equivalent to \Href{Lemma}{lem:hyperplane_meets_facet} and \Href{Lemma}{lem:intersecting_in_barycenter}, respectively.

By  \Href{Lemma}{lem:hyperplane_meets_facet}, all vectors $v \in S$ such that $\span{v}$ intersects $F$  correspond to $F$ and have the same length that we denote by $|v|.$  Then by  \Href{Lemma}{lem:intersecting_in_barycenter}, the span of each of these vectors intersects $F$ in its centroid.
Since the length of the altitude of the pyramid $P_F$ is $1/|v|,$ we have
	$$\vol{k}P_F=\frac{1}{k}\frac{\vol{k-1}F}{|v|}.$$
Hence assertion~\ref{item:first_order_volumes} follows from  \Href{Lemma}{lem:necessary_condition_first_order}.
\end{proof}

\section{Proof of \texorpdfstring{\Href{Theorem}{thm:main_result}}{Theorem 1.3}}
We use the setting of tight frames developed in the previous sections to prove the theorem. More precisely, we use the obtained necessary conditions for a tight frame in $\R^2$ that maximizes~\eqref{eq:main_problem_cube} for $n>k=2$ to prove that the section of the cube generated by the tight frame is a rectangle of area $4C_{\cube}(n,2)=4\sqrt{\lceil n/2\rceil\lfloor n/2\rfloor}$ with the sides of lengths $2\sqrt{\lfloor n/2\rfloor}$ and $2\sqrt{\lceil n/2\rceil}$.

First, let us introduce the notation. 
Let $ S=\{v_1, \dots, v_n \}\in \Omega(n,2)$ be a global maximizer of~\eqref{eq:main_problem_cube} for $k = 2$ and $n>2.$ Clearly, $\cube(S)$ is a centrally symmetric polygon in $\R^2$.  The number of edges of $\cube(S)$ is denoted by  $2f.$ Clearly, $f  \leq n.$ By \Href{Theorem}{thm:maximizer_cyclic}, the polygon
$\cube(S)$ is cyclic; and we denote its circumradius by $R.$ Let $F_1, \dots F_{2f}$ be the edges of $\cube(S)$ enumerated in clockwise direction (that is, edges $F_i$ and $F_{i+f}$ are opposite to each other, $i \in [f]$). We reenumerate the vectors of $S$ in such a way that the vector $v_i$ 
corresponds to  the edge $F_i$ for every $i \in [f].$ The central angle subtended by the  edge $F_i$ is denoted by $2 \varphi_i,$ $i \in [f].$

Clearly, we have the following identities (see \Href{Figure}{fig:planar_section_identities})
\begin{equation}
 \label{eq:sum_phi}
 \varphi_1 + \dots + \varphi_f = \frac{\pi}{2},
 \end{equation}
 \begin{equation}
 \label{eq:R_of_phi}
	 R  \cos \varphi_i = \frac{1}{|v_i|} \quad  \forall i \in [f]  
 \end{equation} 
and
	\begin{equation}\label{eq:area_in_terms_of_angles}
		\area\cube(S)=R^2\sum_{i=1}^f\sin2\varphi_i.
	\end{equation}
Also, we note here that 
\begin{equation}\label{eq:cube_area_lower_bound}
\area \cube(S) \geq 4 C_\cube(n,2)= 4
\sqrt{\left\lceil\frac{n}{2}\right\rceil\left\lfloor\frac{n}{2}\right\rfloor}.
\end{equation}

\begin{figure}[!ht]
	\includegraphics[scale=1.2]{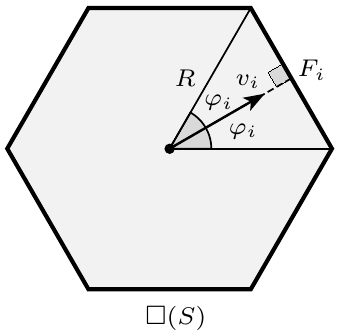}
	\caption{Notation for $\cube(S)$}\label{fig:planar_section_identities}
\end{figure}

There are several steps in the proof. 
We explain the main steps briefly.
In fact, we want to show that the number of edges of a local maximizer is $2f = 4.$ Using the discrete isoperimetric inequality (see below), we obtain an upper bound on $\area\cube(S)$ in terms of $f.$  This upper bound  yields the desired result for  $n \geq 8$ (the bound is less than conjectured volume $4 C_{\cube}(n,2)$ for $n \geq 8$). Finally, we deal with the lower-dimension cases using the necessary conditions obtained earlier. 

The discrete isoperimetric inequality for cyclic polygons says that among all cyclic $f$-gons with fixed circummradius there is a unique maximal area polygon -- the regular $f$-gon. We will use a slightly more general form.
Namely, fixing one or several central angles of a cyclic polygon, its area is maximized when all other central angles are equal. For example, this follows from concavity of the sine function on $[0, \pi].$ In our notation  fixing the central angle $\varphi_i$ and its vertically opposite, we have
\begin{equation}\label{eq:stronger_discrete_isoperimetric_inequality}
R^2 f \sin \frac{\pi}{f} \geq	R^2\left(\sin2\varphi_i + (f-1)\sin\frac{\pi-2\varphi_i}{f-1}\right)\geq 4 C_\cube(n,2).
\end{equation}

\subsection{Step 1.}
\begin{claim}\label{claim:f_2}
	The area of $\cube(S)$ such that $f=2$ is at most $4C_{\cube}(n,2)$. The bound is attained when $\cube(S)$ is a rectangle with the sides of lengths $2\sqrt{\lceil n/2\rceil}$ and $2\sqrt{\lfloor n/2\rfloor}$.
\end{claim}
\begin{proof}
	Since $f=2$, the polygon $\cube(S)$ is an affine square. Hence the claim is an immediate consequence of  \Href{Lemma}{lem:plane_extremizers_cube}.
\end{proof}
Thus, it suffices to prove that $f=2$ for any $n>2$.

\subsection{Step 2.}
\begin{claim}\label{lem:circumradius}
For any $n >2$ the following inequality holds
	\begin{equation}\label{eq:R_upper_bound}
		R^2\leq\frac{n+1}{2}\frac{1}{\cos^2\frac{\pi}{2f}}.
	\end{equation}
\end{claim}
\begin{proof}
Let $\varphi_1$ be the smallest central angle. 
By identity~\eqref{eq:sum_phi},  we have $\cos\varphi_1 \geq \cos \frac{\pi}{2f}.$
Combining this with the leftmost inequality in \eqref{eq:vec_lengths_planar_case} and identity \eqref{eq:R_of_phi},
we obtain
\[
R^2 = \frac{1}{|v_1|^2\cos^2\varphi_1} \leq 
\frac{n+1}{2}\frac{1}{\cos^2\frac{\pi}{2f}}.
\]
\end{proof}
\begin{claim}\label{claim:f_n_bound}
For any $n >2$ the following inequality holds
\begin{equation}
\label{eq:f_n_inequality}
 f\tan\frac{\pi}{2f} \geq\frac{4}{n+1}\sqrt{\left\lceil\frac{n}{2}\right\rceil\left\lfloor\frac{n}{2}\right\rfloor}.
\end{equation}
\end{claim}
\begin{proof}
	By the discrete isoperimetric inequality~\eqref{eq:stronger_discrete_isoperimetric_inequality},  we have
		$$R^2 f \sin \frac{\pi}{f} \geq 4C_{\cube}(n,2).$$
	Combining this with inequalities \eqref{eq:R_upper_bound} and \eqref{eq:cube_area_lower_bound}, we obtain
	\[
	 f \frac{\sin \frac{\pi}{f}}{\cos^2 \frac{\pi}{2f}} \geq \frac{2}{n+1} C_{\cube}(n,2) = \frac{8}{n+1}\sqrt{\left\lceil\frac{n}{2}\right\rceil\left\lfloor\frac{n}{2}\right\rfloor}.
	\]
The claim follows.
\end{proof}
\begin{claim}\label{claim:facet_number_estimation}
	The following bounds on $f$ hold:
	\begin{enumerate}
		\item $f=2$ if $n\geq 8$;
		\item $f\leq 3$ if $n=7$;
		\item $f\leq 4$  if $n=5$.
	\end{enumerate}
\end{claim}
\begin{proof}
	We consider the functions in the left- and right-hand sides of~\eqref{eq:f_n_inequality} as functions of $f$ and $n$ respectively. Set $g(f)=f\tan\frac{\pi}{2f}$ and $h(n)=\frac{4}{n+1}\sqrt{\lfloor n/2\rfloor\lceil n/2\rceil}.$ 
Thus, inequality \eqref{eq:f_n_inequality} takes the form $g(f)\geq h(n).$	
	By routine analysis, we have that 
	$g$ is strictly decreasing  and $h$ is increasing on $\{n \in \N: n \geq 2\}.$
The first two assertions of the claim follows from this and the identity $g(3)=h(7).$
Inequality $f\leq 4$ for $n=5$ follows from the direct computations of $g(5), g(4)$  and $h(5)$ (see \Href{Figure}{fig:tables}).
\begin{figure}[!ht]
	\begin{subfigure}[b]{0.3\textwidth}
	  \centering
		\begin{tabular}{c||c|c|c|c}
		$f$ & $2$ & $3$ & $4$ & $5$\\
		\hline
		$g(f)$ & $2$ & $\sqrt{3}$ & $4\left(\sqrt{2}-1\right)$ & $\sqrt{5\left(5-2\sqrt{5}\right)}$\\
	  \end{tabular}
	\end{subfigure}
	\hspace{8em}
	\begin{subfigure}[b]{0.3\textwidth}
	  \centering
	  \begin{tabular}{c||c|c|c}
		$n$    & $5$  & $6$    & $7$ \\
		\hline
		$h(n)$ & $2\sqrt{6}/3$ & $12 / 7$ & $\sqrt{3}$ \\
	  \end{tabular}
	\end{subfigure}
	\caption{Some values of $g$ and $h$}
	\label{fig:tables}
  \end{figure}
\end{proof}

\Href{Theorem}{thm:main_result} is proven for $n\geq 8$. We proceed with the lower-dimensional cases.

\subsection{Step 3.}
\begin{claim}
For $n=7,$ we have that  $f =2.$
\end{claim}
\begin{proof}
We showed that $f \leq 3$ for $n=7.$ 
Assume that $f=3.$  We see that inequality~\eqref{eq:f_n_inequality} with such values turns into an identity. It follows that $\varphi_1=\varphi_2=\varphi_3=\pi/6$ and $\cube(S)$ is a regular hexagon. Hence the vectors of $S$ are of the same length. Since $\sum\limits_{v \in S} |v|^2 = \operatorname{tr} I_2 =2,$ we conclude that $|v|^2 = 2/7$ and  $R^2 = \frac{1}{|v^2| \cos^2 \pi/6} = 14/3.$ However, the volume of such a hexagon is strictly less than
$4 C_\cube (7,2).$ We conclude that $f=2$ for $n=7.$
\end{proof}
\begin{claim} \label{cl:execptions_cube_planar}
For $n \in \{3,4,6\},$ we have that  $f=2.$
\end{claim}
\begin{proof}
For $n=3,$ the statement is a simple exercise (see \cite{zong}). \Href{Conjecture}{conj:cube_conjecture} was confirmed in \cite{ellips}
for any $n > k \geq 1$ such that $k|n,$ in particular, for $k=2$ and 
$n \in \{4,6\}.$
\end{proof}
\begin{remark} The inequality on the area for $n \in \{4,6\}$ and $k = 2$ is a special case of the leftmost inequality in \eqref{eq:cube_vol_upper_bound} originally proved by K. Ball \cite{ballvolumes}. In \cite{ellips}, the equality cases in this Ball's inequality are described.
\end{remark}

\subsection{Step 4.} 
\begin{claim}\label{claim:n_5_upper_bound_phi}
Let  $n=5$ and either $f=3$ or $f=4.$ 
Then $\varphi_i \leq \pi/4$ for every $i\in[f].$
\end{claim}
\begin{proof}
Assume that there is $i \in [f]$ such that $\varphi_i>\pi/4.$  Thus, $\cos\varphi_i<1/\sqrt2.$ Using identity~\eqref{eq:R_of_phi} for $i$ and for any $j\in[f]$, we have  that
	$$\frac{\cos\varphi_j}{\cos\varphi_i}=\frac{|v_i|}{|v_j|}\leq\sqrt{\frac{n+1}{n-1}}=\sqrt{\frac{3}{2}},$$
where the inequality follows from \Href{Lemma}{lem:squared_lengths_best_bounds}.
Hence
	$\cos\varphi_j\leq\sqrt{3/2}\cos\varphi_i<\sqrt{3}/2$
and, therefore, $\varphi_j>\pi/6.$
This contradicts identity \eqref{eq:sum_phi}:
	$$\frac{\pi}{2}=\varphi_1+\dots+\varphi_f>\frac{\pi}{4}+(f-1)\frac{\pi}{6}>\frac{\pi}{2}.$$
Thus, $\varphi_i\leq\pi/4$ for every $i\in[f].$
\end{proof}
\begin{claim}\label{claim:n_5_lower_bound_phi}
Let  $n=5$ and either $f=3$ or $f=4.$ 
Then $\varphi_i \geq \pi/10$ for every $i\in[f].$
\end{claim}
\begin{proof}
Fix $i \in [f].$ 
By identity~\eqref{eq:R_of_phi} and \Href{Lemma}{lem:squared_lengths_best_bounds}, we have 
$R^2 \leq \frac{3}{2\cos^2 \varphi_i}.$ By inequality~\eqref{eq:stronger_discrete_isoperimetric_inequality}, we get
	\begin{equation*}
		\frac{3}{\cos^2\varphi_i}\left(\sin2\varphi_i+(f-1)\sin\frac{\pi-2\varphi_i}{f-1}\right)\geq 4 C_\cube (5,2) = 4 \sqrt{6}.
	\end{equation*}
The function of  $\varphi_i$ in the left-hand side of this inequality is increasing on  $[0,\pi/2]$. Since the inequality does not hold for $\varphi_i=\pi/10$, we conclude that $\varphi_i$ is necessarily at least $\pi/10$.
\end{proof}

\begin{claim}
For $n=5,$ we have that  $f =2.$
\end{claim}
\begin{proof}
Assume that either $f=3$ or $f=4.$ Denote by $d_i$ the number of vectors in $S$ that correspond to $F_i$. We want to rewrite the inequality of  \Href{Lemma}{lem:necessary_condition_first_order} using the circumradius and the center angle.  Since the length of  edge $F_i$ is $2R\sin\varphi_i$ and by identity~\eqref{eq:R_of_phi}, identity \eqref{eq:necessary_condition_first_order} takes the form
	$$2R^2\sin2\varphi_i
	 =\frac{d_i}{R^2\cos^2\varphi_i}\area\cube(S).$$
Set $q(\varphi)=\cos^2\varphi\sin2\varphi.$ Then for all $i,j\in[f],$ we have
\begin{equation*}
	\frac{d_i}{d_j}=\frac{q(\varphi_i)}{q(\varphi_i)}.
\end{equation*}
Clearly,  there are $i,j\in[f]$ such that $d_i=2$ and $d_j=1.$ 
Therefore, $q(\varphi_i)/q(\varphi_j)=2.$ By \Href{Claim}{claim:n_5_upper_bound_phi} and \Href{Claim}{claim:n_5_lower_bound_phi}, we have that $ \varphi_i, \varphi_j \in [\pi/10,\pi/4].$ By simple computations,  the maximum of $q$ on the segment $[\pi/10,\pi/4]$ is $q(\pi/6) = 3\sqrt{3}/8$ and the minimum is $q(\pi/4) =1/2.$ Hence $\max\limits_{\varphi,\psi\in[\pi/10,\pi/4]} q(\varphi)/q(\psi)=3\sqrt{3}/4<2$ and we come to a contradiction. Thus, $f=2$.
\end{proof}
Thus, we have proved that  for  a maximizer of~\eqref{eq:main_problem_cube}, then $f=2$. By \Href{Claim}{claim:f_2}, the conjectured upper bound for the area of a planar section holds and also is tight. The proof of \Href{Theorem}{thm:main_result} is complete.

\begin{remark}
We used the Ball inequality \eqref{eq:cube_vol_upper_bound} to prove \Href{Theorem}{thm:main_result} for   $n\in\{3,4,6\}.$  However, it can be done by using our approach without the Ball inequality. The proof is technical.  Since it is not of great interest, we do not give a proof.
\end{remark}

\appendix
\section{}\label{sec:appendix}

\begin{proof}[Sketch of the proof of \Href{Lemma}{lem:plane_extremizers_cube}]
Let $H$ be a $k$-dimensional subspace of $\R^n$ such that $C_\cube(n,k)$ is attained. Denote $P = \cube^n \cap H.$   
Since $P$ is an affine $k$-dimensional cube, there are vectors $\{a_1, \dots, a_k\}$ such that $P = \bigcap\nolimits_{i\in[k]} \left(H_{a_i}^+ \cap H_{a_i}^-\right).$

Let $\{v_1, \dots, v_n\}$ be the projection of the vectors of the standard basis onto $H.$  
By the same arguments as in \Href{Lemma}{lem:hyperplane_meets_facet}, the hyperplane $H_{v_i}$ meets the polytope $P$ in a facet of $P$ for every $i \in [n].$ 
Thus, $v_i$ 
	coincides with $\pm a_j$ for a proper sign and $j \in [k].$
	Or, equivalently, we partition $[n]$ into $k$ sets  and $H$ is the solution of a proper system of linear equations constructed as in  
\eqref{enum_card_2} and \eqref{enum_card_3},    except we have not proved that \eqref{enum_card_1} holds yet. Let us prove this assertion. 
Let $d_i$ vectors of the standard basis of $\R^n$ project onto a pair $\pm a_i.$  Therefore, a $k$-tuple of vectors 
$\{\sqrt{d_i} a_i\}_{ i \in [k]}$ is a tight frame. 
Identifying $H$ with $\R^k$ 
and by the assertion \eqref{ass:eq_cond_it3} of  \Href{Lemma}{lem:tight_frame_equiv_cond}, we conclude that 
$a_i$ and $a_j$ are orthogonal whenever $i \neq j.$ Therefore, $|a_i|^2 = \frac{1}{d_i}$ and 
\begin{equation}\label{eq:calc_C_cube}
\vol{k} \cube^n \cap H = 2^k {\sqrt{d_1 \cdot \ldots \cdot d_k}}.
\end{equation}

Suppose $d_i \geq d_j +2$ for some $i, j \in [k].$ Then 
$d_i \cdot d_j  \leq (d_i -1)(d_j+1).$ By this and by \eqref{eq:calc_C_cube}, 
we showed that \eqref{enum_card_1} holds. 

It is easy to see that   there are exactly $n - k \lfloor n/k\rfloor$ of $d_i$'s equal 
$ \lceil n/k\rceil$ and all others $k - (n - k \lfloor n/k\rfloor)$ are equal to $\lfloor n/k\rfloor.$ That is, $C_{\cube}(n,k)$ is given by \eqref{eq:optimal_constant}. This completes the proof.     

\end{proof}

\begin{proof}[Sketch of the proof of \Href{Lemma}{lem:der_det_identity}]
Recall that  the cross product 
 of $k-1$ vectors $\{x_1, \ldots, x_{k-1}\}$ of $\R^k$   is the vector $x$ defined by
$$
	\iprod{x}{y} = \det ( x_1, \ldots, x_{k-1}, y) \quad \mbox{for all}
	\quad  y \in \R^k.
$$
For an ordered $(k-1)$-tuple $L = \{i_1, \dots, i_{k-1} \} \in \binom{[n]}{k-1}$ and a frame $S =\{v_1, \dots, v_n\},$ we use $[v_L]$ to denote the cross product of $v_{i_1}, \dots, v_{i_{k-1}}.$ 

We claim the following property of the tight frames.  \\ \noindent
{\it Let $S =\{v_1, \dots, v_n\}$ be a tight frame in $\R^k$. Then the set of vectors
$\{[v_L]\}_{L \in\binom{[n]}{k-1}}$ is a  tight frame in $\R^k.$} \\ \noindent
We use $\Lambda^k (\R^n)$ to denote the space of exterior $k$-forms on $\R^n.$
By assertion \eqref{ass:eq_cond_it2} of \Href{Lemma}{lem:tight_frame_equiv_cond},
there exists an orthonormal basis  $\{f_i\}_1^n$ of $\R^n$ such that $v_i$ is the orthogonal projection of $f_i$ onto $\R^k,$ for any $i \in [n].$ 
Then the $(k-1)$-form $v_{i_1} \wedge \dots \wedge v_{i_{k-1}}$ is the orthogonal projections of the $(k-1)$-form $f_{i_1} \wedge \dots \wedge f_{i_{k-1}}$ onto $\Lambda^{k-1} (\R^k) \subset \Lambda^{k-1} (\R^n),$ for any ordered $(k-1)$-tuple $L = \{i_1, \dots, i_{k-1} \} \in \binom{[n]}{k-1}.$ By \Href{Lemma}{lem:tight_frame_equiv_cond} and since the $(k-1)$-forms $\{f_{i_1} \wedge \dots \wedge f_{i_{k-1}}\}_{\{i_1, \dots, i_{k-1} \} \in \binom{[n]}{k-1}}$ form an orthonormal basis of $\Lambda^{k-1} (\R^n),$ we have that the set of $(k-1)$-forms  $\{v_{i_1} \wedge \dots \wedge v_{i_{k-1}}\}_{\{i_1, \dots, i_{k-1} \} \in \binom{[n]}{k-1}}$ is a tight frame  in $\Lambda^{k-1} (\R^k).$ Finally, the Hodge star operator maps $v_{i_1} \wedge \dots \wedge v_{i_{k-1}}$ to the cross product of vectors $v_{i_1}, \dots, v_{i_{k-1}}.$ Since the Hodge star is an isometry, the set of  cross products $\{[v_L]\}_{L \in  \in \binom{[n]}{k-1}}$ is a tight frame. The claim is proven.

By linearity of the determinant, it is enough to prove the lemma for 
$\tilde{S}= \{v_1 + t x, v_2, \dots, v_n \}.$ 
Denote $v'_1 = v_1 + tx$ and $v'_i = v_i,$ for $ 2 \leq i \leq n.$

By the Cauchy--Binet formula, we have
\begin{equation}
\label{eq:Binet_Cauchy_A_S}
\det A_{\tilde{S}} = \det  \left( \sum\limits_1^n v'_i \otimes v'_i\right) = 
\sum\limits_{Q \in {\binom{[n]}{k}}} \det  \left( \sum\limits_{i \in Q} v'_i \otimes v'_i\right).
\end{equation}
By the properties of the Gram matrix, we have 
\[
\det  \left( \sum\limits_{1}^k v'_{i_1} \otimes v'_{i_k}\right) =  \left(\det \left( v'_{i_1}, \dots,  v'_{i_k}\right) \right)^2.
\]
By this, by the definition of cross product and by identity \eqref{eq:Binet_Cauchy_A_S},  we obtain 
$$
	\det A_{\tilde{S}} =  1 +  
		2 t \sum\limits_{Q \in {[n] \choose k}, 1 \in Q }
				 \iprod{v_1}{[v_{Q\setminus 1}]}  \iprod{[v_{Q\setminus 1}]}{ x} + o(t).
$$
Since  $\iprod{v_1}{[v_{J}]}  = 0$ for any $J \in  {[n] \choose k-1}$ such that $1 \in J,$ we have that the linear term of the Taylor expansion of $\det A_{\tilde{S}}$ equals  
\[
	2 t \sum\limits_{Q \in {[n] \choose k}, 1 \in Q }
				 \iprod{v_1}{[v_{Q\setminus 1}]} {[v_{Q\setminus 1}], x} =  2t \sum\limits_{L \in {[n] \choose k-1}} \iprod{v_i}{[v_{L}]} 
				 \iprod{[v_{L}]}{ x}.
\]
Since $\{[v_L]\}_{L \in \binom{[n]}{k-1}}$ is a  tight frame in $\R^k$, we have that
\[
 \sum\limits_{L \in {[n] \choose k-1}} \iprod{v_i}{[v_{L}]} 
 \iprod{[v_{L}]}{x} = \iprod{v_i}{x}.
\]
Therefore, 
$$
	\sqrt{\det A_{\tilde{S}}} = \sqrt{\det A_S + 2 t \iprod{v_i}{x} + o(t)} = 1 + t \iprod{v_i}{x} + o(t). 
$$
\end{proof}

\invisiblesection{Acknowledgement}
\subsection*{Acknowledgement}
The authors acknowledge the financial support from the Ministry of Education and Science of the Russian Federation in the framework of MegaGrant no 075-15-2019-1926. G.I. was supported also by the Swiss National Science Foundation grant
200021-179133 and by the Russian Foundation for Basic Research, project 18-01-00036A.

\bibliographystyle{amsalpha}
\bibliography{bibliography}{}
\end{document}